\documentclass[12pt,reqno,tbtags]{amsart}
\pdfoutput=1
\usepackage{amssymb}
\usepackage{txfonts}
\usepackage[colorlinks=true, pdfstartview=FitV, linkcolor=blue,
  citecolor=blue, urlcolor=blue,pagebackref=false]{hyperref}
\usepackage[numeric,initials,nobysame]{amsrefs}
\usepackage{amsmath,amssymb,amsfonts}
\usepackage{enumerate}
\usepackage[all]{xy}

\numberwithin{equation}{section}

\newtheorem{maintheorem}{Theorem}
\newtheorem{maincoro}[maintheorem]{Corollary}
\newtheorem{mainprop}[maintheorem]{Proposition}

\newtheorem{theorem}{Theorem}[section]
\newtheorem*{theorem*}{Theorem}
\newtheorem{lemma}[theorem]{Lemma}

\newtheorem{corollary}[theorem]{Corollary}

\theoremstyle{definition}{
\newtheorem{example}[theorem]{Example}

\newtheorem{question}[theorem]{Question}
}

\theoremstyle{remark}{

\newtheorem*{remark*}{Remark}

}

\newcommand{\R}{\mathbb R}
\newcommand{\N}{\mathbb N}

\newcommand{\E}{\mathbb{E}}
\renewcommand{\P}{\mathbb{P}}
\DeclareMathOperator{\var}{Var}

\renewcommand{\epsilon}{\varepsilon}

\newcommand{\tX}{\tilde{X}}

\newcommand{\one}{\boldsymbol{1}}
\newcommand{\oneb}[1]{\one_{\{ #1 \}}}

\newcommand{\deq}{:=}
\newcommand{\tmix}{t_\textsc{mix}}
\newcommand{\trel}{t_\textsc{rel}}
\newcommand{\tsep}{t_\mathrm{sep}}

\newcommand{\gap}{\text{\tt{gap}}}
\begin{document}
\title{Total variation cutoff in birth-and-death chains}
\date{}

\author{Jian Ding, \thinspace Eyal Lubetzky and Yuval Peres}

\address{Jian Ding\hfill\break
Department of Statistics\\
UC Berkeley\\
Berkeley, CA 94720, USA.}
\email{jding@stat.berkeley.edu}
\urladdr{}

\address{Eyal Lubetzky\hfill\break
Microsoft Research\\
One Microsoft Way\\
Redmond, WA 98052-6399, USA.}
\email{eyal@microsoft.com}
\urladdr{}

\address{Yuval Peres\hfill\break
Microsoft Research\\
One Microsoft Way\\
Redmond, WA 98052-6399, USA.}
\email{peres@microsoft.com}
\urladdr{}
\thanks{Research of J. Ding and Y. Peres was supported in part by NSF grant DMS-0605166.}

\hyphenation{sta-rted}
\begin{abstract}
The \emph{cutoff phenomenon} describes a case where a Markov chain exhibits a sharp transition in its convergence to stationarity. In 1996, Diaconis surveyed this phenomenon, and asked how one could recognize its occurrence in families of finite ergodic Markov chains. In 2004, the third author noted that a necessary condition for cutoff in
a family of reversible chains is
that the product of the mixing-time and spectral-gap tends to infinity, and conjectured that in many settings,
this condition should also be sufficient.
Diaconis and Saloff-Coste (2006) verified this conjecture for continuous-time birth-and-death chains, started at an endpoint,
 with convergence measured in \emph{separation}. It is natural to ask whether the conjecture holds for these chains in the more widely used \emph{total-variation} distance.

In this work, we confirm the above conjecture for all continuous-time or lazy discrete-time birth-and-death chains, with convergence measured via total-variation distance. Namely, if the product of the mixing-time and spectral-gap tends to infinity, the chains exhibit cutoff at the maximal hitting time of the stationary distribution median, with a window of at most the geometric mean between the relaxation-time and mixing-time.

In addition, we show that for any lazy (or continuous-time) birth-and-death chain with stationary distribution $\pi$, the separation $1 - p^t(x,y)/\pi(y)$ is maximized when $x,y$ are the endpoints.
Together with the above results, this implies that total-variation cutoff is equivalent to separation cutoff in any family of such chains.
\end{abstract}

\maketitle

\section{Introduction}

The \emph{cutoff phenomenon} arises when a finite Markov chain converges abruptly to equilibrium. Roughly, this is the case where, over a negligible period of time known as the \emph{cutoff window}, the distance of the chain from the stationary measure drops from near its maximum to near $0$.

Let $(X_t)$ denote an aperiodic irreducible Markov chain on a finite state space $\Omega$ with transition kernel $P(x,y)$,
and let $\pi$ denote its stationary distribution. For any two distributions $\mu,\nu$ on $\Omega$, their \emph{total-variation distance} is defined to be
$$\|\mu-\nu\|_\mathrm{TV} \deq \sup_{A \subset\Omega} \left|\mu(A) - \nu(A)\right| = \frac{1}{2}\sum_{x\in\Omega} |\mu(x)-\nu(x)|~.$$
Consider the worst-case total-variation distance to stationarity at time $t$,
$$ d(t) \deq \max_{x \in \Omega} \| \P_x(X_t \in \cdot)- \pi\|_\mathrm{TV}~,$$
where $\P_x$ denotes the probability given $X_0=x$.
The total-variation \emph{mixing-time} of $(X_t)$, denoted by $\tmix(\epsilon)$ for $0 < \epsilon < 1$, is defined to be
$$ \tmix(\epsilon) \deq \min\left\{t : d(t) \leq \epsilon \right\}~.$$

Next, consider a family of such chains, $(X_t^{(n)})$, each with its corresponding worst-distance from stationarity $d_n(t)$, its mixing-times $\tmix^{(n)}$, etc. We say that this family of chains exhibits \emph{cutoff} iff the following sharp transition in its convergence to stationarity occurs:
\begin{equation}\label{eq-cutoff-def}\lim_{n\to\infty} \frac{\tmix^{(n)}(\epsilon)}{\tmix^{(n)}(1-\epsilon)}=1 \quad\mbox{ for any $0 < \epsilon < 1$}~.\end{equation}

Our main result is an essentially tight bound on the difference between $\tmix(\epsilon)$ and $\tmix(1-\epsilon)$ for general \emph{birth-and-death} chains; a birth-and-death chain has the state space $\{0,\ldots,n\}$ for some integer $n$, and always moves from one state to a state adjacent to it (or stays in place).

We first state a quantitative bound for a single chain, then deduce a cutoff criterion. Let $\gap$ be the spectral-gap of the chain (that is,
$\gap \deq 1-\lambda$ where $\lambda$ is the largest absolute-value of all nontrivial eigenvalues of the transition kernel $P$), and let $\trel\deq \gap^{-1}$ denote the relaxation-time of the chain. A chain is called \emph{lazy} if $P(x,x) \geq \frac{1}{2}$ for all $x\in\Omega$.

\begin{maintheorem}
  \label{thm-1}
For any $0 < \epsilon < \frac{1}{2}$ there exists an explicit $c_\epsilon > 0$ such that every lazy irreducible birth-and-death chain $(X_t)$ satisfies
\begin{equation}\label{eq-thm-1-window-bound}
\tmix(\epsilon) - \tmix(1-\epsilon) \leq c_\epsilon \sqrt{ \trel \cdot \tmix(\mbox{$\frac{1}{4}$})}~.\end{equation}
\end{maintheorem}

As we later show, the above theorem extends to continuous-time chains, as well as to $\delta$-lazy chains, which satisfy $P(x,x) \geq \delta$ for all $x\in\Omega$.

The notion of a cutoff-window relates Theorem \ref{thm-1} to the cutoff phenomenon. A sequence $w_n$ is called a \emph{cutoff window} for a family of chains $(X_t^{(n)})$ if the following holds: $w_n =o\big(\tmix^{(n)}(\frac{1}{4})\big)$, and for any $\epsilon > 0$ there exists some $c_\epsilon > 0$ such that, for all $n$,
\begin{equation}\label{eq-window-defintion} \tmix^{(n)}(\epsilon) - \tmix^{(n)}(1-\epsilon) \leq c_\epsilon w_n~.\end{equation}
Equivalently, if $t_n$ and $w_n$ are two sequences such that $w_n =o(t_n)$, one may define that a sequence of chains exhibits cutoff at $t_n$ with window $w_n$ iff
$$\left\{\begin{array}
  {l}\lim_{\lambda\to\infty} \liminf_{n\to\infty} d_n(t_n - \lambda w_n) = 1~,\\
  \lim_{\lambda \to \infty} \limsup_{n\to\infty} d_n(t_n + \lambda w_n) = 0~.
\end{array}\right.$$
To go from the first definition to the second, take $t_n = \tmix^{(n)}(\frac{1}{4})$.

Once we compare the forms of \eqref{eq-thm-1-window-bound} and \eqref{eq-window-defintion}, it becomes clear that Theorem \ref{thm-1} implies a bound on the cutoff window for any general family of birth-and-death chains, provided that $\trel^{(n)} = o\big(\tmix^{(n)}(\frac{1}{4})\big)$.

Theorem \ref{thm-1} will be the key to establishing the criterion for total-variation cutoff in a general family of birth-and-death chains.

\subsection{Background}
The cutoff phenomenon was first identified for the case of random transpositions on the symmetric group in \cite{DiSh}, and for the case of random walks on the hypercube in \cite{Aldous}. It was given its name by Aldous and Diaconis in their famous paper \cite{AlDi} from 1985, where they showed that the top-in-at-random card shuffling process (repeatedly removing the top card and
reinserting it to the deck at a random position) has such a behavior. Saloff-Coste \cite{SaloffCoste} surveys the cutoff phenomenon for random walks on finite groups.

Though many families of chains are believed to exhibit cutoff, proving the occurrence of this phenomenon is often an extremely challenging task, hence there are relatively few examples for which cutoff has been rigorously shown. In 1996, Diaconis \cite{Diaconis} surveyed the cutoff phenomenon, and asked if one could determine whether or not it occurs in a given family of aperiodic and irreducible finite Markov chains.

In 2004, the third author \cite{Peres} observed that a necessary condition for cutoff in a family of reversible chains is that the product  $\tmix^{(n)}(\frac{1}{4}) \cdot \gap(n)$ tends to infinity with $n$, or equivalently, $\trel^{(n)} = o\big(\tmix^{(n)}(\frac{1}{4})\big)$; see Lemma \ref{lem-tmix-lower-bound}. The third author also conjectured that, in many natural classes of chains,
\begin{equation}
  \label{eq-cutoff-conj}
  \mbox{\emph{Cutoff occurs if and only if }$\trel^{(n)} = o\big(\tmix^{(n)}(\frac{1}{4})\big)$~.}
\end{equation}
In the general case, this condition does not always imply cutoff : Aldous \cite{Aldous2} and Pak (private communication via P. Diaconis) have constructed relevant examples (see also \cite{Chen},\cite{ChSa} and \cite{LPW}). This left open the question of characterizing the classes of chains for which \eqref{eq-cutoff-conj} holds.

One important class is the family of birth-and-death chains; see \cite{DiSa} for many natural examples of such chains. They also occur as the magnetization chain of the mean-field Ising Model (see \cite{DLP},\cite{LLP}).

In 2006, Diaconis and Saloff-Coste \cite{DiSa} verified a variant of the conjecture \eqref{eq-cutoff-conj} for birth-and-death chains, when the convergence
to stationarity is measured in \emph{separation}, that is, according to the decay of $\mathrm{sep}(\P_0(X_t \in \cdot),\pi)$, where
$\mathrm{sep}(\mu,\nu) = \sup_{x\in\Omega} (1-\frac{\mu(x)}{\nu(x)})$.
  Note that, although $\mathrm{sep}(\mu,\nu)$ assumes values in $[0,1]$, it is in fact not a metric (it is not even symmetric).
See, e.g., \cite{AF}*{Chapter 4} for the connections between mixing-times in total-variation and in separation.

More precisely, it was shown in \cite{DiSa} that any family of continuous-time birth-and-death chains, started at $0$, exhibits cutoff in separation if and only if $\trel^{(n)} = o\big(\tsep^{(n)}(\frac{1}{4};0)\big)$, where $\tsep(\epsilon;s) = \min\{t:\mathrm{sep}(\P_s(X_t\in \cdot),\pi)<\epsilon\}$. The proof used
a spectral representation of passage times \cites{KaMc,Keilson} and duality of strong stationary times. Whether \eqref{eq-cutoff-conj} holds with respect to the important and widely used total-variation distance, remained unsettled.

\subsection{Total-variation cutoff}
In this work, we verify the conjecture \eqref{eq-cutoff-conj} for arbitrary birth-and-death chains, with the convergence to stationarity measured in total-variation distance. Our first result, which is a direct corollary of Theorem \ref{thm-1}, establishes this for lazy discrete-time irreducible birth-and-death chains. We then derive versions of this result for continuous-time irreducible birth-and-death chains, as well as for $\delta$-lazy discrete chains (where $P(x,x)\geq \delta$ for all $x\in\Omega$). In what follows, we omit the dependence on $n$ wherever it is clear from the context. Here and throughout the paper, the abbreviation $\tmix$ stands for $\tmix\big(\frac14\big)$.

\begin{maincoro}\label{cor-lazy}
Let $(X^{(n)}_t)$ be a sequence of lazy irreducible birth-and-death chains. Then it exhibits cutoff in total-variation distance iff $\tmix^{(n)} \cdot \gap(n)$ tends to infinity with $n$.
Furthermore, the cutoff window size is at most the geometric mean between the mixing-time and
relaxation time.
\end{maincoro}

As we will later explain, the given bound $\sqrt{\tmix \cdot \trel}$ for the cutoff window is essentially tight, in the following sense.
Suppose that the functions $t_M(n)$ and $t_R(n)\geq 2$ denote the mixing-time and relaxation-time of $(X_t^{(n)})$, a family of irreducible lazy
birth-and-death chains. Then there exists a family $(Y_t^{(n)})$ of such chains with the parameters $\tmix^{(n)}=(1+o(1)) t_M(n)$ and $\trel ^{(n)} = (1+o(1))t_R(n)$ that has a cutoff window of $(\tmix^{(n)} \cdot \trel^{(n)})^{1/2}$. In other words, no better bound on the cutoff window can be given without exploiting additional information on the chains.

Indeed, there are examples where additional attributes of the chain imply a cutoff window of order smaller than $\sqrt{\tmix \cdot \trel}$. For instance, the cutoff window has size $\trel$ for the Ehrenfest urn (see, e.g., \cite{DiMi}) and for the magnetization chain in the mean field Ising Model at high temperature (see \cite{DLP}).

Theorem \ref{thm-delta-lazy}, given in Section \ref{sec:cont-delta-lazy}, extends Corollary \ref{cor-lazy}
to the case of $\delta$-lazy discrete-time chains. We note that this is in fact the setting that corresponds to the magnetization chain in the mean-field Ising Model (see, e.g., \cite{LLP}).

Following is the continuous-time version of Corollary \ref{cor-lazy}.

\begin{maintheorem}\label{thm-cont}
Let $(X_t^{(n)})$ be a sequence of continuous-time birth-and-death chains. Then $(X_t^{(n)})$ exhibits cutoff in total-variation iff $\trel^{(n)} = o(\tmix^{(n)})$,
and the cutoff window size is at most $\sqrt{\tmix^{(n)}(\frac{1}{4}) \cdot \trel^{(n)}}$.
\end{maintheorem}

By combining our results with those of \cite{DiSa} (while bearing in mind the relation between the mixing-times in total-variation and in separation), one can relate worst-case total-variation cutoff in any continuous-time family of irreducible birth-and-death chains, to cutoff in separation started from $0$.
This suggests that total-variation cutoff should be equivalent to separation cutoff in such chains under the original definition of the worst starting point (as opposed to fixing the starting point at one of the endpoints). Indeed, it turns out that for any lazy or continuous-time birth-and-death chain, the separation is always attained by the two endpoints, as formulated by the next proposition.
\begin{mainprop}\label{prop-sep}
Let $(X_t)$ be a lazy (or continuous-time) birth-and-death chain with stationary distribution $\pi$. Then for every integer (resp. real) $t>0$, the separation $1-\P_x(X_t=y)/\pi(y)$ is maximized when $x,y$ are the endpoints.
\end{mainprop}
That is, for such chains, the maximal separation from $\pi$ at time $t$ is simply $1-P^t(0,n)/\pi(n)$ (for the lazy chain with transition kernel $P$) or $1-H_t(0,n)/\pi(n)$ (for the continuous-time chain with heat kernel $H_t$).
As we later show, this implies the following corollary:
\begin{maincoro}\label{cor-tv-sep-worst}
For any continuous-time family of irreducible birth-and-death chains, cutoff in worst-case total-variation distance is equivalent to cutoff in worst-case separation.
\end{maincoro}

Note that, clearly, the above equivalence is in the sense that one cutoff implies the other, yet the cutoff locations need not be equal (and sometimes indeed are not equal, e.g., the Bernoulli-Laplace models, surveyed in \cite{DiSa}*{Section 7}).

The rest of this paper is organized as follows.
The proofs of Theorem \ref{thm-1} and Corollary \ref{cor-lazy} appear in Section \ref{sec:mainproof}. Section \ref{sec:cont-delta-lazy} contains the proofs of the variants of Theorem \ref{thm-1} for the continuous-case (Theorem \ref{thm-cont}) and the $\delta$-lazy case. In Section \ref{sec:separation}, we discuss separation in general birth-and-death chains, and provide the proofs of Proposition \ref{prop-sep} and Corollary \ref{cor-tv-sep-worst}.
The final section, Section \ref{sec:conclusion}, is devoted to concluding remarks and open problems.

\section{Cutoff in lazy birth-and-death chains}\label{sec:mainproof}
In this section we prove the main result, which shows that the condition $\gap \cdot \tmix \to\infty$ is necessary and sufficient for total-variation cutoff in lazy birth-and-death chains.

\subsection{Proof of Corollary \ref{cor-lazy}}
The fact that any family of lazy irreducible birth-and-death chains satisfying
$\tmix \cdot \gap \to\infty$ exhibits cutoff, follows by definition from Theorem \ref{thm-1}, as does the bound $\sqrt{\trel\cdot\tmix}$ on the cutoff window size.

It remains to show that this condition is necessary for cutoff; this is known to hold for any family of reversible Markov chains,  using a straightforward and well known lower bound on $\tmix$ in terms of $\trel$ (cf., e.g., \cite{LPW}). We include its proof for the sake of completeness.
\begin{lemma}\label{lem-tmix-lower-bound}
Let $(X_t)$ denote a reversible Markov chain, and suppose that $\trel \geq 1 + \theta \tmix(\frac{1}{4})$ for some fixed $\theta > 0$. Then for any $0 < \epsilon < 1$
\begin{equation}\label{eq-tmix-lower-bound-tmix}\tmix(\epsilon) \geq \tmix(\mbox{$\frac{1}{4}$}) \cdot \theta \log(1/2\epsilon)~.\end{equation} In particular, $\tmix(\epsilon) / \tmix(\frac{1}{4}) \geq K$ for all $K > 0$ and $\epsilon < \frac{1}{2}\exp(-K/\theta)$.
\end{lemma}
\begin{proof}
Let $P$ denote the transition kernel of $X$, and recall that the fact that $X$ is reversible implies that $P$ is a symmetric operator with respect to $\left<\cdot,\cdot\right>_\pi$ and $\one$ is an eigenfunction corresponding to the trivial eigenvalue $1$.

Let $\lambda$ denote the largest absolute-value of all nontrivial eigenvalues of $P$, and let $f$ be the corresponding
eigenfunction, $P f = \pm \lambda f$, normalized to have $\|f\|_\infty=1$. Finally, let $r$ be the state attaining $|f(r)|=1$. Since $f$ is orthogonal to $\one$, it follows that for any $t$,
\begin{align*}\lambda^t &= \left|(P^t f)(r) - \left<f,\one\right>_\pi\right| \leq \max_{x\in\Omega} \Big| \sum_{y\in \Omega} P^t(x,y) f(y) - \pi(y) f(y) \Big| \\
&\leq \|f\|_\infty  \max_{x\in\Omega} \|P^t(x,\cdot)-\pi\|_1 = 2 \max_{x\in\Omega} \|P^t(x,\cdot)-\pi\|_{\mathrm{TV}}~.\end{align*}
Therefore, for any $0 < \epsilon < 1$ we have
\begin{equation}\label{eq-tmix-lower-bound-trel}\tmix(\epsilon) \geq \log_{1/\lambda} (1/2\epsilon) \geq \frac{\log(1/2\epsilon)}{\lambda^{-1}-1} = (\trel -1)\log(1/2\epsilon)~,\end{equation}
and \eqref{eq-tmix-lower-bound-tmix} immediately follows.
\end{proof}
This completes the proof of Corollary \ref{cor-lazy}.\qed

\subsection{Proof of Theorem \ref{thm-1}}
The essence of proving the theorem lies in the treatment of the regime where $\trel$ is much smaller than $\tmix(\frac{1}{4})$.
\begin{theorem}\label{thm-1-effective}
Let $(X_t)$ denote a lazy irreducible birth-and-death chain, and suppose that $\trel < \epsilon^5 \cdot \tmix(\frac{1}{4})$ for some $0 < \epsilon < \frac{1}{16}$.
Then $$\tmix(4\epsilon) - \tmix(1-2\epsilon) \leq (6/\epsilon)\sqrt{ \trel \cdot \tmix(\mbox{$\frac{1}{4}$})}~.$$
\end{theorem}

\begin{proof}[\textbf{\emph{Proof of Theorem \ref{thm-1}}}]
To prove Theorem \ref{thm-1} from Theorem \ref{thm-1-effective}, let $\epsilon > 0$, and suppose first that $\trel < \epsilon^5 \cdot \tmix(\frac{1}{4})$. If $\epsilon < \frac{1}{64}$, then the above theorem clearly implies
that \eqref{eq-thm-1-window-bound} holds for $c_\epsilon=24/\epsilon$. Since that the left-hand-side of \eqref{eq-thm-1-window-bound} is monotone decreasing in $\epsilon$, this result extends to any value of $\epsilon < \frac{1}{2}$ by choosing $$c_1=c_1(\epsilon) = 24 \max\{ 1/\epsilon, 64\}~.$$
It remains to treat the case where $\trel \geq \epsilon^5 \cdot \tmix(\frac{1}{4})$. In this case, the sub-multiplicativity of the mixing-time (see, e.g., \cite{AF}*{Chapter 2}) gives
\begin{equation}\tmix(\epsilon) \leq \tmix(\mbox{$\frac{1}{4}$})\lceil\mbox{$\frac{1}{2}$}\log_2(1/\epsilon)\rceil\quad \mbox{for any $0<\epsilon < \frac{1}{4}$}~.\end{equation}
In particular, for $\epsilon < \frac{1}{4}$ our assumption on $\trel$ gives
$$\tmix(\epsilon) - \tmix(1-\epsilon) \leq \tmix(\epsilon) \leq
\epsilon^{-5/2} \log_2(1/\epsilon) \sqrt{\trel \cdot \tmix(\mbox{$\frac{1}{4}$})} ~.$$
Therefore, a choice of $$c_2 = c_2(\epsilon) = \max\{ \log_2(1/\epsilon) / \epsilon^{5/2}, 64\}$$ gives \eqref{eq-thm-1-window-bound} for any $\epsilon < \frac{1}{2}$ (the case $\epsilon > \frac{1}{4}$ again follows from monotonicity).

Altogether, a choice of $c_\epsilon = \max\{c_1, c_2\}$ completes the proof.
\end{proof}

In the remainder of this section, we provide the proof of Theorem \ref{thm-1-effective}. To this end, we must first establish several lemmas.

Let $X=X(t)$ be the given (lazy irreducible) birth-and-death chain, and from now on, let $\Omega_n=\{0,\ldots,n\}$ denote its state space. Let $P$ denote the transition kernel of $X$, and let $\pi$ denote its stationary distribution.
Our first argument relates the mixing-time of the chain, starting from various starting positions,
with its hitting time from $0$ to certain quantile states, defined next.
\begin{equation}\label{def-quantile-0}
Q(\epsilon) \deq \min\Big\{k: \sum_{j=0}^{k}\pi(j)\geq
\epsilon\Big\}~,\quad\mbox{where }0<\epsilon<1~.
\end{equation}
Similarly, one may define the hitting times from $n$ as follows:
\begin{equation}\label{def-quantile-n}
\tilde{Q}(\epsilon) \deq \max\Big\{k: \sum_{j=k}^{n}\pi(j)\geq
\epsilon\Big\}~,\quad\mbox{where }0<\epsilon<1~.
\end{equation}
\begin{remark*}
Throughout the proof, we will occasionally need to shift from $Q(\epsilon)$ to $\tilde{Q}(1-\epsilon)$, and vice versa. Though the proof can be written in terms of $Q,\tilde{Q}$, for the sake of simplicity it will be easier to have
the symmetry \begin{equation}\label{eq-Q-symmetry}Q(\epsilon) = \tilde{Q}(1-\epsilon)\mbox{ for almost any }\epsilon > 0~.\end{equation} This is easily achieved by noticing that at most $n$ values of $\epsilon$ do not satisfy \eqref{eq-Q-symmetry} for a given chain $X(t)$ on $n$ states. Hence, for any given countable family of chains, we can eliminate a countable set of all such problematic values of $\epsilon$ and obtain the above mentioned symmetry.
\end{remark*}
Recalling that we defined $\P_k$ to be the probability on the event that the starting position is $k$,
we define $\E_k$ and $\var_k$ analogously. Finally, here and in what follows, let $\tau_k$ denote the hitting-time of the state $k$, that is, $\tau_k\deq \min \{t: X(t)=k\}$.
\begin{lemma}\label{lem-tv-hitting-bound}
For any fixed $0 < \epsilon < 1$ and lazy irreducible birth-and-death chain $X$, the following holds for any $t$:
\begin{align}
  \label{eq-tv-from-0-hitting-bound}
  \| P^t(0,\cdot) - \pi \|_{\mathrm{TV}} &\leq \P_0(\tau_{Q(1-\epsilon)} > t) + \epsilon~,
\end{align}
and for all $k \in \Omega_n$,
\begin{align} \label{eq-tv-from-k-hitting-bound}
\| P^t(k,\cdot) - \pi \|_{\mathrm{TV}} &\leq \P_k(\max\{\tau_{Q(\epsilon)},\tau_{Q(1-\epsilon)}\} > t) + 2\epsilon~.
\end{align}
\end{lemma}
\begin{proof}
Let $X$ denote an instance of the lazy birth-and-death chain starting from a given state $k$, and let $\tX$ denote another instance of the lazy chain starting from the stationary distribution. Consider the following \emph{no-crossing} coupling of these two chains: at each step, a fair coin toss decides which of the two chains moves according to its original (non-lazy) rule. Clearly, this coupling does not allow the two chains to cross one another without sharing the same state first (hence the name for the coupling). Furthermore, notice that by definition, each of the two chains, given the number of coin tosses that went its way, is independent of the other chain. Finally, for any $t$, $\tX(t)$, given the number of coin tosses that went its way until time $t$, has the stationary distribution.

In order to deduce the mixing-times bounds, we show an upper bound on the time it takes $X$ and $\tX$ to coalesce. Consider the hitting time of $X$ from $0$ to $Q(1-\epsilon)$, denoted by $\tau_{Q(1-\epsilon)}$. By the above argument, $\tX(\tau_{Q(1-\epsilon)})$ enjoys the stationary distribution, hence by the definition of $Q(1-\epsilon)$,
$$\P\left(\tX(\tau_{Q(1-\epsilon)}) \leq X(\tau_{Q(1-\epsilon)})\right) \geq 1-\epsilon~.$$
Therefore, by the property of the no-crossing coupling, $X$ and $\tX$ must have coalesced by time $\tau_{Q(1-\epsilon)}$ with probability at least $1-\epsilon$. This implies \eqref{eq-tv-from-0-hitting-bound}, and it remains to prove \eqref{eq-tv-from-k-hitting-bound}. Notice that the above argument involving the no-crossing coupling, this time with $X$ starting from $k$, gives
$$\P\left(\tX(\tau_{Q(\epsilon)}) \geq X(\tau_{Q(\epsilon)})\right) \geq 1-\epsilon~,$$
and similarly,
$$\P\left(\tX(\tau_{Q(1-\epsilon)}) \leq X(\tau_{Q(1-\epsilon)})\right) \geq 1-\epsilon~.$$
Therefore, the probability that $X$ and $\tX$ coalesce between the times $\tau_{Q(\epsilon)}$ and $\tau_{Q(1-\epsilon)}$ is at least $1-2\epsilon$, completing the proof.
\end{proof}

\begin{corollary}\label{cor-mixing-time-order}
Let $X(t)$ be a lazy irreducible birth-and-death chain on $\Omega_n$. The following holds for any $0 < \epsilon < \frac{1}{16}$:
\begin{equation}\label{eq-mixing-order-hitting-bound}
 \tmix(\mbox{$\frac{1}{4}$}) \leq 16 \max\left\{\E_0\tau_{Q(1-\epsilon)},\E_n\tau_{Q(\epsilon)}\right\}~.
\end{equation}
\end{corollary}
\begin{proof}
Clearly, for any source and target states $x,y\in\Omega_n$, at least one of the endpoints $s\in\{0,n\}$ satisfies $\E_s \tau_y \geq \E_x \tau_y$ (by the definition of the birth-and-death chain). Therefore, if $T$ denotes the right-hand-side of \eqref{eq-mixing-order-hitting-bound}, then
$$\P_x(\max\{\tau_{Q(\epsilon)},\tau_{Q(1-\epsilon)}\} \geq T) \leq
\P_x(\tau_{Q(\epsilon)} \geq T) + \P_x(\tau_{Q(1-\epsilon)} \geq T) \leq \frac{1}{8}~,$$
where the last transition is by Markov's inequality.
The proof now follows directly from \eqref{eq-tv-from-k-hitting-bound}.
\end{proof}
\begin{remark*} The above corollary shows that the order of the mixing time is at most $\max\{\E_0\tau_{Q(1-\epsilon)},\E_n\tau_{Q(\epsilon)}\}$.
  It is in fact already possible (and not difficult) to show that the mixing time has this order \emph{precisely}.
  However, our proof only uses the order of the mixing-time as an upper-bound, in order to finally deduce a stronger result: this mixing-time is asymptotically equal to the above maximum of the expected hitting times.
\end{remark*}
Having established that the order of the mixing-time is at most the expected hitting time of $Q(1-\epsilon)$ and $Q(\epsilon)$ from the two endpoints of $\Omega_n$, assume here and in what follows, without loss of generality, that $\E_0\tau_{Q(1-\epsilon)}$ is at least $\E_n\tau_{Q(\epsilon)}$. Thus, \eqref{eq-mixing-order-hitting-bound} gives
\begin{equation}\label{eq-tmix-hitting-from-0-bound}
 \tmix(\mbox{$\frac{1}{4}$}) \leq 16 \cdot \E_0\tau_{Q(1-\epsilon)}~\mbox{ for any $0 < \epsilon < \frac{1}{16}$}~.
\end{equation}

A key element in our estimation is a result of
Karlin and McGregor \cite{KaMc}*{Equation (45)}, reproved by Keilson \cite{Keilson}, which
represents hitting-times for birth-and-death chains in continuous-time as a sum of independent exponential variables
 (see \cite{Fill},\cite{DiMi}, \cite{Fill2} for more on this result). The discrete-time version of this result was given by Fill \cite{Fill}*{Theorem 1.2}.
\begin{theorem}[\cite{Fill}]\label{thm-Keilson}
Consider a discrete-time birth-and-death chain with transition kernel $P$ on the state space $\{0,\ldots,d\}$
started at $0$. Suppose that $d$ is an absorbing state, and
suppose that the other birth probabilities $p_i$, $0 \leq i \leq d-1$, and death probabilities
$q_i$, $1 \leq i \leq d-1$, are positive. Then the absorption time in state $d$ has probability generating
function
\begin{equation}\label{eq-Keilson-pgf}
u\mapsto \prod_{j=0}^{d-1}
\Big[\frac{(1-\theta_{j})u}{1-\theta_{j}u}\Big]~,
\end{equation}
where $-1 \leq \theta_{j} < 1$ are the $d$ non-unit eigenvalues of $P$.
Furthermore, if $P$ has nonnegative eigenvalues then the
absorption time in state $d$ is distributed as the sum of $d$ independent geometric
random variables whose failure probabilities are the non-unit eigenvalues of $P$.
\end{theorem}
The above theorem provides means of establishing the
concentration of the passage time from left to right of a chain, where the target (right end) state
is turned into an absorbing state. Since we are interested in the hitting time from one end to a given state
(namely, from $0$ to $Q(1-\epsilon)$), it is clearly equivalent to consider the chain where the target state
is absorbing. We thus turn to handle the hitting time of an absorbing end of a chain starting from the other end. The following lemma will infer its concentration from Theorem \ref{thm-Keilson}.
\begin{lemma} \label{lem-hitting-concentration}
Let $X(t)$ be a lazy irreducible birth-and-death chain on the state space $\{0,\ldots,d\}$, where
$d$ is an absorbing state, and let $\gap$ denote its spectral gap. Then
$\var_{0}\tau_{d} \leq \left(\E_0 \tau_d\right)/\gap$.
\end{lemma}
\begin{proof}
Let $\theta_0 \geq \ldots \geq \theta_{d-1}$ denote the $d$ non-unit eigenvalues of the transition kernel of $X$.
Recalling that $X$ is a lazy irreducible birth-and-death chain, $\theta_i \geq 0$ for all $i$, hence the
second part of Theorem \ref{thm-Keilson} implies that
$ \tau_d \sim \sum_{i=0}^{d-1} Y_i$,
where the $Y_i$-s are independent geometric random variables with means $1/(1-\theta_i)$.
Therefore,
\begin{align}\label{eq-exp-var-hitting}
\E_{0} \tau_d = \sum _{i=0}^{d-1}
\frac{1}{1-\theta_{i}}~,\quad \var_{0} \tau_d =  \sum_{i=0}^{d-1} \frac{\theta_i}{\left(1-\theta_{i}\right)^2}~,
\end{align}
which, using the fact that $\theta_0 \geq \theta_i$ for all $i$, gives
\begin{align*}
\var_{0}\tau_d & \leq
\frac{1}{1-\theta_{0}} \sum_{i=0}^{d-1}
\frac{1}{1-\theta_{i}} = \frac{\E_{0}\tau_d}{\gap}~,
\end{align*}
as required.
\end{proof}
As we stated before, the hitting time of a state in our original chain has the same distribution as
the hitting time in the modified chain (where this state is set to be an absorbing state). In order to
derive concentration from the above lemma, all that remains is to relate the spectral gaps of these two chains.
This is achieved by the next lemma.
\begin{lemma}\label{lem-gap-gap-inequality}
Let $X(t)$ be a lazy irreducible birth-and-death chain, and $\gap$ be its spectral gap.
Set $0 < \epsilon < 1$, and let $\ell= Q(1-\epsilon)$. Consider the
modified chain $Y(t)$, where $\ell$ is turned into an absorbing state, and let $\gap|_{[0,\ell]}$
denote its spectral gap. Then $\gap|_{[0,\ell]} \geq \epsilon \cdot \gap$.
\end{lemma}
\begin{proof}
 By \cite{AF}*{Chapter 3, Section 6}, we have
\begin{equation}\label{eq-gap}
\gap = \mathop{\min_{f\;:\;\E_\pi f=0}}_{f \not\equiv 0}
\frac{\left<(I-P)f,f\right>_\pi}{\left<f,f\right>_\pi} =
\mathop{\min_{f\;:\;\E_\pi f=0}}_{f \not\equiv 0} \frac{1}{2}
\frac{\sum_{i,j}\left(f(i)-f(j)\right)^{2} P(i,j)\pi(i)}{\sum_{i} f(i)^{2}\pi(i)}~.
\end{equation}
Observe that $\gap|_{[0, \ell]}$ is precisely $1-
\lambda$, where $\lambda$ is the largest eigenvalue
of $P|_{\ell}$, the principal sub-matrix on the first $\ell$ rows and columns, indexed
by $\{0,\ldots,\ell-1\}$ (notice that this sub-matrix is strictly sub-stochastic, as $X$ is irreducible).
Being a birth-and-death chain, $X$ is reversible, that is,
$$ P_{i j} \pi(i) = P_{j i} \pi(j)\mbox{ for any }i,j~.$$
Therefore, it is simple to verify that $P|_{\ell}$ is a symmetric operator on
$\R^{\ell}$ with respect to the inner-product $\left<\cdot,\cdot\right>_\pi$; that is,
$\left<P|_{\ell} x,y\right>_\pi = \left<x,P|_{\ell} y\right>_\pi$ for every $x,y\in\R^{\ell}$, and hence
the Rayleigh-Ritz formula holds (cf., e.g., \cite{Halmos}), giving
$$ \lambda = \mathop{\max_{x \in\R^{\ell}}}_{x\neq 0} \frac{\left<P|_{\ell} x,x\right>_\pi}{\left<x,x\right>_\pi}~.$$
It follows that
\begin{align}
\gap|_{[0,\ell]} = 1-\lambda &= \mathop{\min_{f\;:\; f \nequiv 0}}_{f(k)=0 \;\forall k\geq \ell}
\frac{\sum_{i=0}^{n}\left(f(i)-
\sum_{j=0}^{n}P(i,j)f(j)\right)
f(i) \pi(i)}{\sum_{i=0}^{n}{f(i)^2 \pi(i)}}\nonumber\\
&= \mathop{\min_{f \;:\;  f\nequiv 0}}_{f(k)=0 \;\forall k\geq \ell}
\frac{1}{2}\frac{\sum_{0\leq i,j\leq n}\left(f(i)-
f(j)\right)^{2}P(i,j)\pi(i)}{\sum_{i=0}^{n}{f(i)^2 \pi(i)}}
~,\label{eq-gap-l}
\end{align}
where the last equality is by the fact that $P$ is stochastic.

Observe that \eqref{eq-gap} and \eqref{eq-gap-l} have similar forms,
and for any $f$ (which can also be treated as a random
variable) we can write $\tilde{f}= f - \E_{\pi} f$ such
that $\E_{\pi} \tilde{f}=0$. Clearly,
$$\left(f(i)-f(j)\right)^{2} P(i,j)\pi(i) = (\tilde{f}(i)-\tilde{f}(j))^{2} P(i,j)\pi(i),$$
hence in order to compare $\gap$ and $\gap|_{[0,\ell]}$, it will
suffice to compare the denominators of \eqref{eq-gap} and
\eqref{eq-gap-l}. Noticing that
$$\var_{\pi}(f)= \sum_{i} \tilde{f}(i)^{2}\pi(i)~,~\mbox{ and }
\E_{\pi} f^2 = \sum_{i} f(i)^{2}\pi(i)~,$$
we wish to bound the ratio between the above two terms.
Without loss of generality, assume that $\E_{\pi} f
=1$. Then every $f$ with $f(k)=0$ for all $k\geq
\ell$ satisfies
\begin{equation*}
\frac{\E_{\pi}f^2}{\pi(f\neq0)} = \E_{\pi} \left[f^2 \mid f\neq0\right]\geq
\left(\E_{\pi}\left[f \mid f\neq0\right]\right)^{2} = \left(\pi\left(f\neq0\right)\right)^{-2}~,
\end{equation*}
and hence
\begin{equation}
\frac{1}{\E_{\pi}f^2} \leq \pi(f\neq0) \leq 1-\epsilon~,
\end{equation}
where the last inequality is by the definition of $\ell$ as $Q(1-\epsilon)$.
Once again, using the fact that $\E_{\pi}f=1$, we deduce that
\begin{equation}
\frac{\var_{\pi}f}{\E_{\pi}f^2}= 1-
\frac{1}{\E_{\pi}f^2} \geq \epsilon~.
\end{equation}
Altogether, by the above discussion
on the comparison between \eqref{eq-gap} and \eqref{eq-gap-l}, we conclude
 that $\gap|_{[0,\ell]}\geq \epsilon \cdot\gap$.
\end{proof}
Combining Lemma \ref{lem-hitting-concentration} and Lemma \ref{lem-gap-gap-inequality}
yields the following corollary:
\begin{corollary}
\label{cor-quantile-hitting-concentration}
 Let $X(t)$ be a lazy irreducible birth-and-death chain on $\Omega_n$, let $\gap$ denote its spectral-gap,
 and $0 < \epsilon < 1$. The following holds:
\begin{equation}\label{eq-quantile-hitting-concentration}
\var_0 \tau_{Q(1-\epsilon)} \leq \frac{\E_{0}\tau_{Q(1-\epsilon)}}{\epsilon \cdot \gap}~.
\end{equation}
\end{corollary}
\begin{remark*}
  The above corollary implies the following statement: whenever $\gap \cdot \E_{0}\tau_{Q(1-\epsilon)}\to\infty$
   with $n$, the hitting-time $\tau_{Q(1-\epsilon)}$ is concentrated, as $\var_0 \tau_{Q(1-\epsilon)} = o\big( (\E_0\tau_{Q(1-\epsilon)})^2\big)$. This is essentially the case
   under the assumptions of Theorem \ref{thm-1-effective} (which include a lower bound on $\gap \cdot \tmix(\frac{1}{4})$ in terms of $\epsilon$),
    as we already established in \eqref{eq-tmix-hitting-from-0-bound} that $\E_0 \tau_{Q(1-\epsilon)} \geq \frac{1}{16}\tmix(\frac{1}{4})$.
\end{remark*}

Recalling the definition of cutoff and the relation between the mixing
time and hitting times of the quantile states, we expect that the behaviors of
$\tau_{Q(\epsilon)}$ and $\tau_{Q(1-\epsilon)}$ would be roughly the same;
this is formulated in the following lemma.

\begin{lemma}\label{lem-hitting-ratio}
Let $X(t)$ be a lazy irreducible birth-and-death chain on $\Omega_n$,
and suppose that for some $0 < \epsilon < \frac{1}{16}$ we have
$\trel < \epsilon^4 \cdot \E_0 \tau_{Q(1-\epsilon)}$.
Then for any fixed $\epsilon \leq \alpha < \beta \leq 1-\epsilon$:
\begin{equation}\label{eq-commute-time}
  \E_{Q(\alpha)} \tau_{Q(\beta)} \leq \frac{3}{2\epsilon} \sqrt{ \trel \cdot \E_0 \tau_{Q(\frac{1}{2})}}~.
\end{equation}
\end{lemma}
\begin{proof}
Since by definition, $\E_{Q(\epsilon)}\tau_{Q(1-\epsilon)} \geq \E_{Q(\alpha)} \tau_{Q(\beta)}$
(the left-hand-side can be written as a sum of three independent hitting times, one of which being
the right-hand-side), it suffices to show \eqref{eq-commute-time} holds for $\alpha=\epsilon$ and $\beta=1-\epsilon$.

Consider the random variable $\nu$, distributed according to the restriction of the stationary distribution $\pi$
to $[Q(\epsilon)]\deq\{0,\ldots,Q(\epsilon)\}$, that is:
\begin{equation}
\nu(k) \deq \frac{\pi(k)}{\pi([Q(\epsilon)])}\oneb{[Q(\epsilon)]}~,
\end{equation}
and let $w \in \R^{\Omega_n}$ denote the vector $ w \deq \oneb{[Q(\epsilon)]}/\pi([Q(\epsilon)])$.
As $X$ is reversible, the following holds for any $k$:
$$P^t(\nu,k) = \sum_i P^t(i,k) \pi(i) w(i) = (P^t w)(k) \cdot \pi(k)~.$$
Thus, by the definition of the total-variation distance (for a finite space):
\begin{align*}
  \|P^{t}(\nu, \cdot)- \pi(\cdot)\|_\mathrm{TV} &=
\frac{1}{2}\sum_{k=0}^n \pi(k)\left| \left(P^t w\right)(k) - 1 \right| = \frac{1}{2}\| P^t (w - \one) \|_{L^1(\pi)}\\
&\leq \frac{1}{2}\| P^t (w - \one) \|_{L^2(\pi)}~,
\end{align*}
where the last inequality follows from the Cauchy-Schwartz
inequality. As $w-\one$ is orthogonal to $\one$ in the inner-product space $\left<\cdot,\cdot\right>_{L^2(\pi)}$, we deduce that
$$\| P^t (w-\one)\|_{L^2(\pi)} \leq \lambda_2^t  \|w-\one\|_{L^2(\pi)}~,$$
where $\lambda_2$ is the second largest eigenvalue of $P$. Therefore,
\begin{align*}\|P^{t}(\nu, \cdot)- \pi(\cdot)\|_\mathrm{TV} &\leq \frac{1}{2}\lambda_2^t \|w-\one\|_{L^2(\pi)} =\frac{1}{2}\lambda_2^t \sqrt{\left(1/\pi([Q(\epsilon)])\right)-1} \leq \frac{\lambda_2^t }{2\sqrt{\epsilon}} ~,\end{align*}
where the last inequality is by the fact that $\pi([Q(\epsilon)]) \geq \epsilon$ (by definition).
Recalling that $\trel = \gap^{-1} = 1/(1-\lambda_2)$, define
$$t_\epsilon = \left\lceil \mbox{$\frac32$}\log(1/\epsilon)\trel\right\rceil~,$$
and notice that, as $\epsilon < \frac1{16}$ and $\trel \geq 1$, we have $\frac12 \log(1/\epsilon) \trel \geq 1$,
and so
$$t_\epsilon \leq 2\log(1/\epsilon)\trel~.$$
Since $\log(1/x)\geq 1-x$ for all $x\in(0,1]$, it follows that
$\lambda_2^{t_\epsilon} \leq \epsilon^{3/2}$, thus
\begin{align}\label{eq-tv-epsilon/2}
\left\|P^{t_\epsilon}(\nu, \cdot)- \pi(\cdot)\right\|_\mathrm{TV} \leq \epsilon/2~.
\end{align}
We will next use a second moment argument to obtain an upper bound on the expected commute time. By \eqref{eq-tv-epsilon/2} and the definition of the total-variation distance,
\begin{align*}
\P_{\nu}(\tau_{Q(1-\epsilon)} \leq t_\epsilon)  \geq \epsilon - \|P^{t_\epsilon}(\nu, \cdot)- \pi(\cdot)\|_\mathrm{TV}
 \geq \epsilon / 2~,
\end{align*}
whereas the definition of $\nu$ as being supported by the range $[Q(\epsilon)]$ gives
\begin{align*}
\P_{\nu}(\tau_{Q(1-\epsilon)} \leq t_\epsilon)
\leq \P_{Q(\epsilon)}(\tau_{Q(1-\epsilon)} \leq t_\epsilon)\leq \frac{\var_{Q(\epsilon)}\tau_{Q(1-\epsilon)}}{\left|
\E_{Q(\epsilon)}\tau_{Q(1-\epsilon)}-t_\epsilon\right|^{2}}~.
\end{align*}
Combining the two,
\begin{equation}\label{eq-tau-var-mix}
\E_{Q(\epsilon)}\tau_{Q(1-\epsilon)} \leq
t_\epsilon + \sqrt{\frac{2}{\epsilon}\var_{Q(\epsilon)}\tau_{Q(1-\epsilon)}}.
\end{equation}
Recall that starting from $0$, the hitting time to point
$Q(1-\epsilon)$ is exactly the sum of the hitting time from $0$
to $Q(\epsilon)$ and the hitting time from $Q(\epsilon)$ to
$Q(1-\epsilon)$, where both these hitting times are independent.
Therefore,
\begin{equation}\label{eq-tau-var-var}
\var_{Q(\epsilon)} \tau_{Q(1-\epsilon)} \leq \var_{0}
\tau_{Q(1-\epsilon)}~.
\end{equation}
By \eqref{eq-tau-var-mix} and \eqref{eq-tau-var-var} we get
\begin{align}\label{eq-commute-bound-1}
\E_{Q(\epsilon)}\tau_{Q(1-\epsilon)} &\leq
t_\epsilon + \sqrt{\frac{2}{\epsilon}\var_{0}\tau_{Q(1-\epsilon)}} \nonumber\\
&\leq 2\log(1/\epsilon)\trel + (1/\epsilon)\sqrt{2 \trel \cdot \E_0 \tau_{Q(1-\epsilon)}}~,
\end{align}
where the last inequality is by Corollary \ref{cor-quantile-hitting-concentration}.

We now wish to rewrite the bound \eqref{eq-commute-bound-1} in terms of $\trel \cdot \E_0\tau_{Q(\frac{1}{2})}$
using our assumptions on $\trel$ and $\E_0\tau_{Q(\frac{1}{2})}$.
First, twice plugging in the fact that
$\trel < \epsilon^4 \cdot \E_0\tau_{Q(1-\epsilon)}$ yields
\begin{align}
\E_{Q(\epsilon)}\tau_{Q(1-\epsilon)} 
&\leq
 \left(2 \epsilon^3 \log(1/\epsilon) + \sqrt{2}\right)\epsilon \cdot \E_0 \tau_{Q(1-\epsilon)} \nonumber \\
 &\leq \mbox{$\frac{3}{2}$}\epsilon\cdot \E_0 \tau_{Q(1-\epsilon)} \label{eq-commute-bound-2}~,
\end{align}
where in the last inequality we used the fact that $\epsilon < \frac{1}{16}$.
In particular,
\begin{align*}
\E_0 \tau_{Q(1-\epsilon)} &\leq \E_0 \tau_{Q(\frac{1}{2})} + \E_{Q(\epsilon)} \tau_{Q(1-\epsilon)}
\leq \E_0 \tau_{Q(\frac{1}{2})} + \mbox{$\frac{3}{2}$}\epsilon\cdot \E_0 \tau_{Q(1-\epsilon)}~,
\end{align*}
and after rearranging,
\begin{equation}\label{eq-tau-q(1-epsilon)-bound}
\E_0 \tau_{Q(1-\epsilon)} \leq \left(\E_0 \tau_{Q(\frac{1}{2})} \right) / \left( 1-\mbox{$\frac{3}{2}$}\epsilon\right)~.
\end{equation}
Plugging this result back in \eqref{eq-commute-bound-1}, we deduce that
\begin{equation*}
\E_{Q(\epsilon)}\tau_{Q(1-\epsilon)} \leq
2\log(1/\epsilon)\cdot \trel + \frac{1}{\epsilon}\sqrt{\frac{2\trel \cdot \E_0 \tau_{Q(\frac{1}{2})}}{1-\mbox{$\frac{3}{2}$}\epsilon}}~.
\end{equation*}
A final application of the fact $\trel < \epsilon^4 \cdot \E_0 \tau_{Q(1-\epsilon)}$, together with \eqref{eq-tau-q(1-epsilon)-bound} and the fact that $\epsilon < \frac{1}{16}$,
gives
\begin{align}
\E_{Q(\epsilon)}\tau_{Q(1-\epsilon)} &\leq  \bigg( \frac{2\epsilon^{2} \log(1/\epsilon) + \sqrt{2}/\epsilon }{\sqrt{1-\frac{3}{2}\epsilon}} \bigg) \sqrt{\trel \cdot \E_0 \tau_{Q(\frac{1}{2})}}\nonumber\\
&\leq \frac{3}{2\epsilon}\sqrt{ \trel \cdot \E_0 \tau_{Q(\frac{1}{2})}}~,
 \label{eq-commute-bound-3}
\end{align}
as required.
\end{proof}
We are now ready to prove the main theorem.
\begin{proof}[\textbf{\emph{Proof of Theorem \ref{thm-1-effective}}}]
Recall our assumption (without loss of generality)
\begin{equation}
  \label{eq-E0-tau(Q(1-epsilon))-wlog-assumption}
  \E_0 \tau_{Q(1-\epsilon)} \geq \E_n \tau_{Q(\epsilon)}~,
\end{equation}
and define what would be two ends of the cutoff window:
\begin{align*}
\left\{\begin{array}{ll}
  t^- = t^-(\gamma) \deq \Big\lfloor \; \E_0 \tau_{Q(\frac{1}{2})} - \gamma \sqrt{\trel \cdot \E_0 \tau_{Q(\frac{1}{2})}} \; \Big\rfloor~,\\
 t^+ = t^+(\gamma) \deq  \Big\lceil \; \E_0 \tau_{Q(\frac{1}{2})} + \gamma \sqrt{\trel \cdot \E_0 \tau_{Q(\frac{1}{2})}} \; \Big\rceil~.
 \end{array}\right.
\end{align*}
For the lower bound, let $0 < \epsilon < \frac{1}{16}$; combining \eqref{eq-tmix-hitting-from-0-bound} with the assumption that $\trel \leq \epsilon^5 \cdot \tmix(\frac{1}{4})$ gives
\begin{equation}\label{eq-trel-hitting-time-prob}\trel \leq 16 \epsilon^5  \cdot \E_0\tau_{Q(1-\epsilon)} < \epsilon^4 \cdot \E_0\tau_{Q(1-\epsilon)}~.\end{equation}
Thus, we may apply Lemma \ref{lem-hitting-ratio} to get
\begin{align*}\E_0 \tau_{Q(\epsilon)} &\geq \E_0 \tau_{Q(\frac{1}{2})}
- \E_{Q(\epsilon)}\tau_{Q(1-\epsilon)} \geq \E_0 \tau_{Q(\frac{1}{2})}
- \frac{3}{2\epsilon}\sqrt{ \trel \cdot \E_0\tau_{Q(\frac{1}{2})}} ~.\end{align*}
Furthermore, recalling Corollary \ref{cor-quantile-hitting-concentration}, we also have
\begin{align*}
\var_0 \tau_{Q(\epsilon)} &\leq \frac{1}{1-\epsilon}\trel \cdot \E_0\tau_{Q(\epsilon)} \leq 2\trel \cdot \E_0\tau_{Q(\frac{1}{2})}~.
\end{align*}
Therefore, by Chebyshev's inequality, the following holds for any $\gamma > \frac{3}{2\epsilon}$:
$$ \| P^{t^-}(0,\cdot)-\pi \|_\mathrm{TV} \geq 1 -\epsilon - \P_0(\tau_{Q(\epsilon)} \leq t^-)
\geq 1 - \epsilon - 2\left(\gamma-\frac{3}{2\epsilon}\right)^{-2} ~,$$
and a choice of $\gamma =2/\epsilon$ implies that (with room to spare, as $\epsilon < \frac{1}{16}$)
\begin{equation}\label{eq-tmix-final-lower-bound} \tmix(1-2\epsilon) \geq \E_0 \tau_{Q(\frac{1}{2})} - (2/\epsilon) \sqrt{\trel \cdot \E_0 \tau_{Q(\frac{1}{2})}}~.\end{equation}

The upper bound will follow from a similar argument. Take $0 < \epsilon < \frac{1}{16}$ and recall
that $\trel < \epsilon^4\cdot \E_0\tau_{Q(1-\epsilon)}$. Applying Corollary \ref{cor-quantile-hitting-concentration} and Lemma \ref{lem-hitting-ratio} once more (with \eqref{eq-tau-q(1-epsilon)-bound} as well as \eqref{eq-E0-tau(Q(1-epsilon))-wlog-assumption} in mind) yields:
\begin{align}\E_n \tau_{Q(\epsilon)} &\leq \E_0 \tau_{Q(1-\epsilon)} \leq \E_0 \tau_{Q(\frac{1}{2})}
+ \frac{3}{2\epsilon}\sqrt{ \trel \cdot \E_0\tau_{Q(\frac{1}{2})}} ~,\label{eq-tau-Qs}\\
\var_0 \tau_{Q(1-\epsilon)} &\leq (1/\epsilon)\trel \cdot \E_0 \tau_{Q(1-\epsilon)}  \leq \frac{\trel \cdot \E_0\tau_{Q(\frac{1}{2})}}{\epsilon(1-\frac{3}{2}\epsilon)} \leq (2/\epsilon)\trel \cdot \E_0\tau_{Q(\frac{1}{2})}~,\nonumber\\
\var_n \tau_{Q(\epsilon)} &\leq (1/\epsilon)\trel \cdot \E_n \tau_{Q(\epsilon)} \leq (2/\epsilon) \trel \cdot \E_0\tau_{Q(\frac{1}{2})}~.\nonumber
\end{align}
Hence, combining Chebyshev's inequality with \eqref{eq-tv-from-k-hitting-bound} implies that for all $k$ and $\gamma > \frac{3}{2\epsilon}$,
\begin{align*} \| P^{t^+}(k,\cdot) - \pi \|_{\mathrm{TV}} &\leq 2\epsilon +
\P_0(\tau_{Q(1-\epsilon)} > t^+) + \P_n(\tau_{Q(\epsilon)} > t^+) \\
&\leq 2\epsilon  + \frac{4}{\epsilon} \left(\gamma-\frac{3}{2\epsilon}\right)^{-2} ~. \end{align*}
Choosing $\gamma = \frac{35}{12\epsilon}$ we therefore get (with room to spare)
\begin{equation*}\tmix(4\epsilon) \leq \Big\lceil\;\E_0 \tau_{Q(\frac{1}{2})} + \frac{35}{12\epsilon} \sqrt{\trel \cdot \E_0 \tau_{Q(\frac{1}{2})}}\;\Big\rceil~.\end{equation*}
Note that, $Q\big(\frac12\big) > 0$, since otherwise $\E_0\tau_{Q(\frac12)} = 0$ and thus \eqref{eq-tau-Qs} would imply that
$\E_n \tau_{Q(\epsilon)} = \E_0\tau_{Q(1-\epsilon)}=0$. Indeed, in that case, we would get $Q(1-\epsilon)=0$ and yet $Q(\epsilon)=n$,
and therefore $n = 0$, turning the statement of the theorem to be trivially true. It follows that $\trel \cdot \E_0\tau_{Q(\frac12)} \geq 1$, and combining this with the fact that $\epsilon < \frac1{16}$ we conclude that
\begin{equation}\label{eq-tmix-upper-lower-bound} \tmix(4\epsilon) \leq \E_0 \tau_{Q(\frac{1}{2})} + (3/\epsilon) \sqrt{\trel \cdot \E_0 \tau_{Q(\frac{1}{2})}}~.\end{equation}
We have thus established the cutoff window in terms of $\trel$ and $\E_0\tau_{Q(\frac{1}{2})}$, and it remains to write it in terms of $\trel$ and $\tmix$. To this end, recall that \eqref{eq-tau-q(1-epsilon)-bound} implies that
$$ \trel < \epsilon^4 \cdot \E_0\tau_{Q(1-\epsilon)} \leq \frac{\epsilon^4}{1-\frac{3}{2}\epsilon}\E_0\tau_{Q(\frac{1}{2})}~,$$
hence \eqref{eq-tmix-final-lower-bound} gives the following for any $\epsilon < \frac{1}{16}$:
\begin{align}\label{eq-tmix-fine-upper-bound}
\tmix\mbox{$\left(\frac{1}{4}\right)$} &\geq \bigg(1 - \frac{2\epsilon}{\sqrt{1-\frac{3}{2}\epsilon}} \bigg) \cdot \E_0 \tau_{Q(\frac{1}{2})} \geq \frac{5}{6} \E_0 \tau_{Q(\frac{1}{2})}~.
\end{align}
Altogether, \eqref{eq-tmix-final-lower-bound}, \eqref{eq-tmix-upper-lower-bound} and \eqref{eq-tmix-fine-upper-bound} give
\begin{align*} \tmix(4\epsilon) - \tmix(1-2\epsilon) &\leq (5/\epsilon)\sqrt{\trel \cdot \E_0\tau_{Q(\frac{1}{2})} } \leq
(6/\epsilon)\sqrt{ \trel \cdot \tmix\big(\mbox{$\frac{1}{4}$}\big)} ~,\end{align*}
completing the proof of the theorem.
\end{proof}

\subsection{Tightness of the bound on the cutoff window}
The bound $\sqrt{\tmix \cdot \trel}$ on the size of the cutoff window, given in Corollary \ref{cor-lazy}, is essentially tight in the following sense. Suppose that $t_M(n)$ and $t_R(n) \geq 2$ are the mixing-time $\tmix(\frac{1}{4})$ and relaxation-time $\trel$ of a family $(X_t^{(n)})$ of lazy irreducible birth-and-death chains that exhibits cutoff. For any fixed $\epsilon > 0$, we construct a family $(Y_t^{(n)})$ of such chains satisfying
\begin{equation}\label{eq-tight-window-construction}
  \left\{\begin{array}
    {ll} (1-\epsilon)t_M \leq \tmix^{(n)}(\frac{1}{4}) \leq (1+\epsilon)t_M~,\\
    |\trel^{(n)} - t_R| \leq \epsilon~,
  \end{array}\right.
\end{equation}
and in addition, having a cutoff window of size $(\tmix^{(n)} \cdot \trel^{(n)})^{1/2}$.

Our construction is as follows: we first choose $n$ reals in $[0,1)$, which would serve as the nontrivial eigenvalues of our chain: any such sequence can be realized as the nontrivial eigenvalues of a birth-and-death chain with death probabilities all zero, and an absorbing state at $n$. Our choice of eigenvalues will be such that $\tmix^{(n)} = (\frac{1}{2}+o(1))t_M$, $\trel^{(n)} = \frac{1}{2}t_R$ and the chain will exhibit cutoff with a window of $\sqrt{t_M \cdot t_R}$. Finally, we perturb the chain to make it irreducible, and consider its lazy version to obtain \eqref{eq-tight-window-construction}.

First, notice that $t_R = o(t_M)$ (a necessary condition for the cutoff, as given by Corollary \ref{cor-lazy}). Second, if a family of chains has mixing-time and relaxation-time $t_M$ and $t_R$ respectively, then the cutoff point is without loss of generality the expected hitting time from $0$ to some state $m$ (namely, for $m=Q(\frac{1}{2})$); let $h_m$ denote this expected hitting time. Theorem \ref{thm-Keilson} gives $$h_m = \E_0 \tau_m \leq \E_0 \tau_n \leq n\cdot t_R~.$$

Setting $\epsilon > 0$, we may assume that $t_R \geq 2(1+\epsilon)$ (since $t_R \geq 2$, and a small additive error is permitted in \eqref{eq-tight-window-construction}).
Set $K = \frac{1}{2}h_m / t_R$, and define the following sequence of eigenvalues $\{\lambda_i\}$: the first $\lfloor K \rfloor$ eigenvalues will be equal to $\lambda \deq 1 - 2/t_R$, and the remaining eigenvalues will all have the value $\lambda'$, such that the sum $\sum_{i=1}^n 1/(1-\lambda_i)$ equals $\frac{1}{2}h_m$ (our choice of $K$ and the fact that $h_m \leq n t_R$ ensures that $\lambda' \leq \lambda$). By Theorem \ref{thm-Keilson}, the birth-and-death chain with absorbing state in $n$ which realizes these eigenvalues satisfies:
\begin{align*}
\left\{\begin{array}
  {l}\tmix^{(n)} = (1+o(1))\E_0 \tau_n = (\frac{1}{2}+o(1)) t_M~,\\ \trel^{(n)} = \frac{1}{2}t_R~,\\
\var_0 \tau_n \geq \lfloor K \rfloor \frac{\lambda}{(1-\lambda)^2} \geq \frac{\epsilon+o(1)}{8(1+\epsilon)} t_M \cdot t_R~,
\end{array}\right.\end{align*}
where in the last inequality we merely considered the contribution of the first $\lfloor K \rfloor$ geometric random variables to the variance. Continuing to focus on the sum of these $\lfloor K\rfloor$ i.i.d. random variables, and recalling that $K\to\infty$ with $n$ (by the assumption $t_R=o(t_M)$), the Central-Limit-Theorem implies that
$$ \lim_{n\to\infty}\P_0 (\tau_n - \E_0 \tau_n > \gamma \sqrt{t_M \cdot t_R}) \geq c(\gamma,\epsilon) > 0\quad\mbox{ for any }\gamma > 0~.$$
Hence, the cutoff window of this chain has order at least $\sqrt{t_M \cdot t_R}$.

Clearly, perturbing the transition kernel to have all death-probabilities equal some $\epsilon'$ (giving an irreducible chain), shifts every eigenvalue by at most $\epsilon'$ (note that $\tau_n$ from $0$ has the same distribution if $n$ is an absorbing state). Finally, the lazy version of this chain has twice the values of $\E_0\tau_n$ and $\trel$, giving the required result \eqref{eq-tight-window-construction}.

\section{Continuous-time chains and $\delta$-lazy discrete-time chains}\label{sec:cont-delta-lazy}
In this section, we discuss the versions of Corollary \ref{cor-lazy} (and Theorem \ref{thm-1-effective}) for the cases of either continuous-time chains (Theorem \ref{thm-cont}),
or $\delta$-lazy discrete-time chains (Theorem \ref{thm-delta-lazy}). Since
the proofs of these versions follow the original arguments almost entirely, we describe only the modifications
required in the new settings.

\subsection{Continuous-time birth-and-death chains}
In order to prove Theorem \ref{thm-cont},
recall the definition of the heat-kernel of a continuous-time chain as $H_t(x,y)\deq\P_x\left(X_t=y\right)$,
rewritten in matrix-representation as $H_t = \mathrm{e}^{t(P-I)}$ (where $P$ is the transition kernel
of the chain).

It is well known (and easy) that if $H_t,\widetilde{H}_t$ are the heat-kernels corresponding to the continuous-time chain
and the lazy continuous-time chain, then $H_t = \widetilde{H}_{2t}$ for any $t$. This follows immediately
 from the next simple and well-known matrix-exponentiation argument shows:
\begin{equation}
  \label{eq-cont-lazy-recaling}
  H_t = \mathrm{e}^{t(P-I)} = \mathrm{e}^{2t(\frac{P+I}{2} - I)} = \widetilde{H}_{2t}~.
\end{equation}
Hence, it suffices to show cutoff for the lazy continuous-time chains. We therefore need to simply adjust
the original proof dealing with lazy irreducible chains, from the discrete-time case to the continuous-time case.

 The first modification is in the proof of Lemma \ref{lem-tv-hitting-bound}, where
a no-crossing coupling was constructed for the discrete-time chain. Clearly, no such coupling is required
for the continuous case, as the event that the two chains cross one another at precisely the same time
now has probability $0$.

To complete the proof, one must show that the statement of Corollary \ref{cor-quantile-hitting-concentration} still holds; indeed, this follows from the fact that the hitting time $\tau_{Q(1-\epsilon)}$ of the discrete-time chain
is concentrated, combined with the concentration of the sum of the exponential variables that determine the timescale
of the continuous-time chain.

\subsection{Discrete-time $\delta$-lazy birth-and-death chains}
\begin{theorem}\label{thm-delta-lazy}
 Let $(X_t^{(n)})$ be a family of discrete-time $\delta$-lazy birth-and-death chains, for some fixed $\delta > 0$. Then $(X_t^{(n)})$ exhibits cutoff in total-variation iff $\trel^{(n)} = o(\tmix^{(n)})$,
and the cutoff window size is at most $\sqrt{\tmix^{(n)}(\frac{1}{4}) \cdot \trel^{(n)}}$.
\end{theorem}
\begin{proof}
In order to extend Theorem \ref{thm-1-effective} and
 Corollary \ref{cor-lazy} to $\delta$-lazy chains, notice that there are precisely two locations where
their proof rely on the fact that the chain is lazy. The first location is the construction of the no-crossing coupling in the proof of Lemma \ref{lem-tv-hitting-bound}. The second location is the fact that all eigenvalues are non-negative in the application of Theorem \ref{thm-Keilson}.

Though we can no longer construct a no-crossing coupling, Lemma \ref{lem-tv-hitting-bound} can be mended as follows: recalling that $P(x,x) \geq \delta$ for all $x\in\Omega_n$, define $P' = \frac{1}{1-\delta}(P - \delta I)$, and notice that $P'$ and $P$ share the same stationary distribution (and hence define the same quantile states $Q(\epsilon)$ and $Q(1-\epsilon)$ on $\Omega_n$). Let $X'$ denote a chain which has the transition kernel $P'$, and $X$ denote its coupled appropriate lazy version: the number of steps it takes $X$ to perform the corresponding move of $X'$ is an independent geometric random variable with mean $1/(1-\delta)$.

Set $p = 1-\delta(1-2\epsilon)$, and condition on the path of the chain $X'$, from the starting point and until this chain completes $T = \lceil \log_p\epsilon \rceil$ rounds from $Q(\epsilon)$ to $Q(1-\epsilon)$, back and forth. As argued before, as $X$ follows this path, upon completion of each commute time from $Q(\epsilon)$ to $Q(1-\epsilon)$ and back, it has probability at least $1-2\epsilon$ to cross $\tX$. Hence, by definition, in each such trip there is a probability of at least $\delta(1-2\epsilon)$ that $X$ and $\tX$ coalesce. Crucially, these events are independent, since we pre-conditioned on the trajectory of $X'$. Thus, after $T$ such trips, the $X$ and $\tX$ have a probability of at least $1-\epsilon$ to meet, as required.

It remains to argue that the expressions for the expectation and variance of the hitting-times, which were derived from Theorem \ref{thm-Keilson}, remain unchanged when moving from the $\frac{1}{2}$-lazy setting to $\delta$-lazy chains. Indeed, this follows directly from the expression for the probability-generating-function, as given in \eqref{eq-Keilson-pgf}.
\end{proof}

\bigskip
\section{Separation in birth-and-death chains}\label{sec:separation}
In this section, we provides the proofs for Proposition \ref{prop-sep} and Corollary \ref{cor-tv-sep-worst}.

Let $(X_t)$ be an ergodic birth-and-death chain on $\Omega_n=\{0,\ldots,n\}$, with a transition kernel $P$ and stationary distribution $\pi$. Let $d_\mathrm{sep}(t;x)$ denote the separation of $X$, started from $x$, from $\pi$, that is $$d_\mathrm{sep}(t;x) \deq \max_{y\in\Omega_n}\left(1-P^t(x,y)/\pi(y)\right)~.$$
According to this notation, $d_\mathrm{sep}(t) \deq \max_{x\in\Omega_n} d_\mathrm{sep}(t;x)$ measures separation from the worst starting position.

The chain $X$ is called \emph{monotone} iff $P_{i,i+1} + P_{i+1,i}\leq 1$ for all $i< n$. It is well known (and easy to show) that if $X$ is monotone, then the likelihood ratio $P^t(0,k)/\pi(k)$ is monotone decreasing in $k$ (see, e.g., \cite{DiFi}). An immediate corollary of this fact is that the separation of such a chain from the stationary distribution is the same for the two starting points $\{0,n\}$. We provide the proof of this simple fact for completeness.
\begin{lemma}\label{lem-sep-monotone-chain-endpoint}
Let $P$ be the transition kernel of a monotone birth-and-death chain on $\Omega_n=\{0,\ldots,n\}$. If $f:\Omega_n\to\R$ is a monotone increasing (decreasing) function, so is $P f$. In particular,
\begin{equation}\label{eq-monotone-likelihood} P^t(k,0) \geq P^t(k+1,0)~\mbox{ for any $t \geq 0$ and $0 \leq k < n$ }.\end{equation}
\end{lemma}
\begin{proof}
Let $\{p_i\}$, $\{q_i\}$ and $\{r_i\}$ denote the birth, death and holding probabilities of the chain respectively, and for convenience, let $f(x)$ be $0$ for any $x\notin\Omega_n$. Assume without loss of generality that $f$ is increasing (otherwise, one may consider $-f$). In this case, the following holds for every $0 \leq x < n$:
\begin{align*}
  Pf(x) &= q_x f(x-1) + r_x f(x) + p_x f(x+1)  \\
  &\leq (1-p_x) f(x) + p_x f(x+1)~,
\end{align*}
and
\begin{align*}
  Pf(x+1) &= q_{x+1} f(x) + r_{x+1} f(x+1) + p_{x+1} f(x+2)  \\
  &\geq q_{x+1} f(x) + (1-q_{x+1}) f(x+1)~.
\end{align*}
Therefore, by the monotonicity of $f$ and the fact that $p_x+q_{x+1}\leq 1$ we obtain that $Pf(x) \leq Pf(x+1)$, as required.

Finally, the monotonicity of the chain implies that $P^t(\cdot,0)$ is monotone decreasing for $t=1$, hence the above argument immediately implies that this is the case for any integer $t \geq 1$.
\end{proof}
By reversibility, the following holds for any monotone birth-and-death chain with transition kernel $P$ and stationary distribution $\pi$:
\begin{equation}\label{eq-monotone-likelihood-2} \frac{P^t(0,k)}{\pi(k)} \geq \frac{P^t(0,k+1)}{\pi(k+1)}~\mbox{ for any $t \geq 0$ and $0 \leq k < n$ }.\end{equation}
In particular, the maximum of $1-P^t(0,j)/\pi(j)$ is attained at $j=n$, and
the separation is precisely the same when starting at either of the two endpoints:
\begin{corollary}\label{cor-sep-symmetry}
 Let $(X_t)$ be a monotone irreducible birth-and-death chain on $\Omega_n=\{0,\ldots,n\}$ with transition kernel $P$ and stationary distribution $\pi$. Then for any integer $t$,
 $$\mathrm{sep}\left(P^t(0,\cdot),\pi\right) = 1 - \frac{P^t(0,n)}{\pi(n)} = 1-\frac{P^t(n,0)}{\pi(0)}
 =\mathrm{sep}\left(P^t(n,\cdot),\pi\right)~.$$
\end{corollary}

Since lazy chains are a special case of monotone chains, the relation \eqref{eq-cont-lazy-recaling} between lazy and non-lazy continuous-time chains gives an analogous statement for continuous-time irreducible birth-and-death chains. That is, for any real $t> 0$,
 $$\mathrm{sep}\left(H_t(0,\cdot),\pi\right) = 1 - \frac{H_t(0,n)}{\pi(n)} = 1-\frac{H_t(n,0)}{\pi(0)}
 =\mathrm{sep}\left(H_t(n,\cdot),\pi\right)~,$$
 where $H_t$ is the heat-kernel of the chain, and $\pi$ is its stationary distribution.


Unfortunately, when attempting to generalize Lemma \ref{lem-sep-monotone-chain-endpoint} (and Corollary \ref{cor-sep-symmetry}) to an arbitrary starting point, one finds that it is no longer the case that the worst separation involves one of the endpoints, even if the chain is monotone and irreducible. This is demonstrated next.
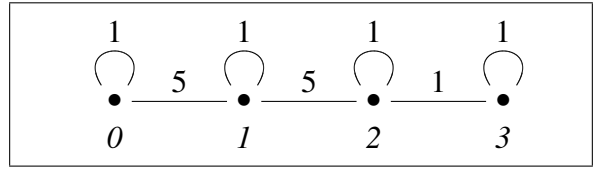
\begin{figure}
$$
\fbox{
\qquad
\def\labelstyle{\mbox}
\SelectTips{cm}{}
\xymatrix @+1pc@R=0pt {
 \bullet  \ar@{-}[r]^{5} \ar@(ru,lu)@{-}[]_1
& \bullet \ar@{-}[r]^{5} \ar@(ur,ul)@{-}[]_1
& \bullet \ar@{-}[r]^{1} \ar@(ur,ul)@{-}[]_1
& \bullet \ar@(ur,ul)@{-}[]_1 \\
\emph{0} & \emph{1} & \emph{2} & \emph{3}
}
\qquad
}
$$
\caption{A monotone irreducible birth-and-death chain where worst separation may not involve the endpoints. Edge weights denote the conductances (see Example \ref{example-monotone-worst-starting-pos}).}\label{fig-sep-chain}
\end{figure}
\begin{example}\label{example-monotone-worst-starting-pos}
Let $P$ and $\pi$ denote the transition kernel and stationary distribution of the birth-and-death chain on the state space $\Omega_3=\{0,1,2,3\}$, given in Figure \ref{fig-sep-chain}. It is easy to verify that this chain is monotone and irreducible, and furthermore, that the following holds:
\begin{align*}
&  \min_{y\in\Omega_3} \frac{P^2(1,y)}{\pi(y)}~\mbox{ is attained solely at $y=2$,}\\
&  \min_{x,y\in\Omega_3} \frac{P^3(x,y)}{\pi(y)}~\mbox{ is attained solely at $x=y=1$.}
\end{align*}Thus, when starting from an interior point, the worst separation might not be attained by an endpoint, and in addition, the overall worst separation may not involve the endpoints at all.
\end{example}
However, as we next show, once we replace the monotonicity requirement with the stricter assumption that the chain is \emph{lazy}, it turns out that the above phenomenon can no longer occur.

The approach that
led to the following result relied on \emph{maximal couplings} (see, e.g., \cite{Griffeath}, \cite{Pitman} and \cite{Goldstein}, and also \cite{CouplingMethod}*{Chapter III.3}). We provide a straightforward proof for it, based on an inductive argument.
\begin{lemma}\label{lem-sep-lazy-chain-middle}
 Let $P$ be the transition kernel of a lazy birth-and-death chain. Then for any unimodal non-negative $f:\Omega \to \R^+$, the function $P f$ is also unimodal. In particular, for any integer $t$, all columns of $P^t$ are unimodal.
\end{lemma}
\begin{proof}
Let $\{p_i\}$, $\{q_i\}$ and $\{r_i\}$ be the birth, death and holding probabilities of the chain respectively,
 and for convenience, define $f(i)$ to be $0$ for $i\in \N \setminus\Omega$.
Let $m\in\Omega$ be a state achieving the global maximum of $f$, and set $g = P f$.

For every $0 < x < m$, the unimodality of $f$ implies that
\begin{align*}
  g(x) &=  q_x f(x-1) + r_x f(x) + p_x f(x+1) \\&\geq q_x f(x-1) + (1-q_x) f(x)~,
  \end{align*}
and similarly,
  \begin{align*}
  g(x-1) &= q_{x-1} f(x-2) + r_{x-1} f(x-1) + p_{x-1} f(x) \\
  &\leq (1-p_{x-1})f(x-1) + p_{x-1}f(x)~.
\end{align*}
Therefore, by the monotonicity of the chain, we deduce that $g(x) \geq g(x-1)$. The same argument shows that
for every $m < y < n$ we have $g(y) \geq g(y+1)$.

As $g$ is increasing on $\{0,\ldots,m-1\}$ and decreasing on $\{m+1,\ldots,n\}$,
unimodality will follow from showing that $g(m) \geq \min\left\{g(m-1),g(m+1)\right\}$ (the global maximum of $g$ would then be attained at $m'\in\{m-1,m,m+1\}$).
To this end, assume without loss of generality that $f(m-1) \geq f(m+1)$. The following holds:
\begin{align*}
  g(m) &=  q_m f(m-1) + r_m f(m) + p_m f(m+1) \\&\geq r_m f(m) + (1-r_m) f(m+1)~,
\end{align*}
and
  \begin{align*}
  g(m+1) &= q_{m+1} f(m) + r_{m+1} f(m+1) + p_{m+1} f(m+2) \\
  &\leq q_{m+1}f(m) + (1-q_{m+1}) f(m+1)~.
\end{align*}
Thus, the laziness of the chain implies that $g(m) \geq g(m+1)$, as required.
\end{proof}

By reversibility, Lemma \ref{lem-sep-lazy-chain-middle} has the following corollary:
\begin{corollary}\label{cor-sep-unimodal}
 Let $(X_t)$ be a lazy and irreducible birth-and-death chain on the state space $\Omega_n=\{0,\ldots,n\}$, with transition kernel $P$ and stationary distribution $\pi$.
 Then for any $s\in\Omega_n$ and any integer $t\geq 0$, the function $f(x)\deq P^t(s,x)/\pi(x)$ is unimodal.
\end{corollary}
\begin{remark*}
The maximum of the unimodal function $f(x)$ in Corollary \ref{cor-sep-unimodal} need not be located at $x=s$, the starting point of the chain. This can be demonstrated, e.g., by the biased random walk.
\end{remark*}
Proposition \ref{prop-sep} will immediately follow from the above results.
\begin{proof}[\textbf{\emph{Proof of Proposition \ref{prop-sep}}}]
We begin with the case where $(X_t)$ is a lazy birth-and-death chain, with transition kernel $P$.
Let $s\in \Omega_n$ be a starting position which maximizes $d_\mathrm{sep}(t)$. Then by Corollary \ref{cor-sep-unimodal}, $d_\mathrm{sep}(t)$ is either equal to $1-P^t(s,0)/\pi(0)$ or to $1-P^t(s,n)/\pi(n)$. Consider the first case (the second case is treated by the exact same argument); by reversibility,
$$d_\mathrm{sep}(t) = 1 - \frac{P^t(0,s)}{\pi(s)} \leq 1 - \frac{P^t(0,n)}{\pi(n)}~,$$
where the last inequality is by Lemma \ref{lem-sep-monotone-chain-endpoint}. Therefore, the endpoints of $X$ assume the worst separation distance at every time $t$.

To show that $d_\mathrm{sep}(t) = 1-H_t(0,n)/\pi(n)$ in the continuous-time case, recall that
$$H_t(x,y) = \P_x(X_t = y) = \E \left[P^{N_t}(x,y)\right] = \sum_k P^k(x,y) \P(N_t = k)~,$$
where $P$ is the transition kernel of the corresponding discrete-time chain, and $N_t$ is a Poisson random variable with mean $t$. Though $P^k$ has unimodal columns for any integer $k$, a linear combination of the matrices $P^k$ does not necessarily maintain this property. We therefore consider a variant of the process, where $N_t$ is approximated by an appropriate binomial variable.

Fix $t > 0$, and for any integer $m \geq 2t$ let $N'_t(m)$ be a binomial random variable with parameters $\mathrm{Bin}(m,t/m)$. Since $N'_t(m)$ converges in distribution to $N_t$, it follows that $H'_t(m) \deq \E \left[P^{N'_t(m)}\right]$ converges to $H_t$ as $m\to\infty$. Writing $N'_t(m)$ as a sum of independent indicators $\{B_i:i=1,\ldots,m\}$ with success probabilities $t/m$, and letting
$Q \deq \left(1-\frac{t}{m}\right)I + \frac{t}{m}P$, we have
$$ H'_t(m) = \E\left[P^{\sum_{i=1}^m B_i}\right] = Q^m~.$$
Note that for every $m \geq 2t$, the transition kernel $Q$ corresponds to a lazy birth-and-death chain, thus Lemma \ref{lem-sep-lazy-chain-middle} ensures that $H'_t(m)$ has unimodal columns for every such $m$. In particular, $H_t = \lim_{m\to\infty}H'_t(m)$ has unimodal columns. This completes the proof.
\end{proof}

\begin{proof}[\textbf{\emph{Proof of Corollary \ref{cor-tv-sep-worst}}}]
By Theorem \ref{thm-cont}, total-variation cutoff (from the worst starting position) occurs iff $\trel = o\big(\tmix(\frac{1}{4})\big)$. Combining
Proposition \ref{prop-sep} with \cite{DiSa}*{Theorem 5.1} we deduce that separation cutoff (from the worst starting point) occurs if and only if $\trel = o\big(\tsep(\frac{1}{4})\big)$ (where $\tsep(\epsilon)=\max_x \tsep(\epsilon ; x)$ is the minimum $t$ such that $\max_x \mathrm{sep}(H_t(x, \cdot),\pi) \leq \epsilon$).

Therefore, the proof will follow from the well known fact that $\tsep(\frac{1}{4})$ and $\tmix(\frac{1}{4})$ have the same order. One can obtain this fact, for instance, from Lemma 7 of \cite{AF}*{Chapter 4}, which states that (as the chain is reversible)
$$ \bar{d}(t) \leq d_\mathrm{sep}(t) ~,~\mbox{ and }~ d_\mathrm{sep}(2t) \leq 1-\left(1-\bar{d}(t)\right)^2~,$$
where $\bar{d}(t) \deq \max_{x,y\in\Omega}\left\|\P_x(X_t\in\cdot)-\P_y(X_t\in\cdot)\right\|_\mathrm{TV}$. Combining this with the sub-multiplicativity of $\bar{d}(t)$, and the fact that $d(t)\leq \bar{d}(t)\leq 2d(t)$ (see Definition 3.1 in \cite{AF}*{Chapter 4}), we obtain that for any $t$,
$$ d(t) \leq d_\mathrm{sep}(t)~,~\mbox{ and }~ d_\mathrm{sep}(8t) \leq 2 \bar{d}(4t) \leq 32 \left(d(t)\right)^4 ~.$$
This in turn implies that $\frac{1}{8} \tsep(\frac{1}{4}) \leq \tmix(\frac{1}{4}) \leq \tsep(\frac{1}{4})$, as required.
\end{proof}

\section{Concluding remarks and open problems}\label{sec:conclusion}
\begin{itemize}
%
%

\item As stated in Corollary \ref{cor-tv-sep-worst}, our results on continuous-time birth-and-death chains, combined with those of \cite{DiSa}, imply that
cutoff in total-variation distance is equivalent to separation cutoff for such chains.
This raises the following question:
\begin{question}
Let $(X_t^{(n)})$ denote a family of irreducible reversible Markov chains, either in continuous-time or in lazy discrete-time. Is it true that there is cutoff in separation iff there is cutoff in total-variation distance (where the distance in both cases is measured from the worst starting position)?
\end{question}

\item One might assume that the cutoff-criterion \eqref{eq-cutoff-conj} also holds for close variants of birth-and-death chains.  For that matter, we note that Aldous's example of a family of reversible Markov chains, which satisfies $\trel^{(n)}=o\big(\tmix^{(n)}(\frac{1}{4})\big)$ and yet does not exhibit cutoff,
    can be written so that each of its chains is a biased random walk on a cycle. In other words, it suffices
    that a family of
    birth-and-death chains permits the one extra transition between states $0$ and $n$, and
    already the cutoff criterion \eqref{eq-cutoff-conj} ceases to hold.
\medskip
\item Finally, it would be interesting to characterize the cutoff criterion in additional natural families of ergodic Markov chains.
\begin{question}
  Does \eqref{eq-cutoff-conj} hold for the family of lazy simple random walks on vertex transitive bounded-degree graphs?
\end{question}
\end{itemize}

\section*{Acknowledgments}

We thank Persi Diaconis, Jim Fill, Jim Pitman and Laurent Saloff-Coste for valuable comments on an early draft, as well as
an anonymous referee for useful suggestions.


\end{document}